\documentclass[11pt]{amsart}
\usepackage{amssymb,amsthm,amsmath,amstext}
\usepackage{mathdots}
\usepackage{array}
\usepackage{varwidth}
\usepackage{xcolor}
\usepackage{multirow}
\usepackage{enumitem}
\usepackage{colonequals}
\usepackage{cases}
\usepackage{graphicx}
\usepackage{mathrsfs} 

\usepackage[left=1.0in,top=1in,right=1.0in,bottom=1in]{geometry}

\usepackage[all]{xy}

\renewcommand{\P}{\mathbb{P}}
\newcommand{\g}[2]{g^{#1}_{#2}}
\DeclareMathOperator{\Pic}{Pic}

\newcommand{\floor}[1]{\left\lfloor #1 \right\rfloor}

\newcommand{\OO}{\mathscr{O}}

\newcommand{\inv}{^{-1}}

\usepackage[backref=page]{hyperref}
\hypersetup{pdftitle={Curves on Brill--Noether special K3 surfaces},pdfauthor={Richard Haburcak}} 
\hypersetup{colorlinks=true,linkcolor=blue,anchorcolor=blue,citecolor=blue}

\usepackage[nameinlink]{cleveref}

\newtheoremstyle{thmbld}{\topsep}{\topsep}{}{}{\itshape}{}{0.5em}{}

\newtheorem{theorem}{Theorem}[section]
\newtheorem{lemma}[theorem]{Lemma}
\newtheorem{prop}[theorem]{Proposition}
\newtheorem{cor}[theorem]{Corollary}
\newtheorem{conj}[theorem]{Conjecture}

\theoremstyle{definition}
\newtheorem{defn}[theorem]{Definition}

\theoremstyle{definition}
\newtheorem{remark}[theorem]{Remark}

\newtheorem{theoremintro}{Theorem}

\newtheorem{questionintro}{Question}

\newcolumntype{A}{w{c}{0.7cm}}

\title{Curves on Brill--Noether special K3 surfaces} 

\author{Richard Haburcak}
\address{Department of Mathematics\\%
		Dartmouth College\\%
		Kemeny Hall\\%
		Hanover, NH 03755}
\email{richard.haburcak.gr@dartmouth.edu}

\begin{document}

\begin{abstract}
Mukai showed that projective models of Brill--Noether general
polarized K3 surfaces of genus $6-10$ and $12$ are obtained as linear
sections of projective homogeneous varieties, and that their
hyperplane sections are Brill--Noether general curves. In general, the question, raised by Knutsen, and attributed to Mukai,
of whether the Brill--Noether generality of any polarized K3 surface
$(S,H)$ is equivalent to the Brill--Noether generality of smooth
curves $C$ in the linear system $|H|$, is still open.  Using
Lazarsfeld--Mukai bundle techniques, we answer this question in the
affirmative for polarized K3 surfaces of genus $\leq 19$, which
provides a new and unified proof even in the genera where Mukai models
exist.
\end{abstract}
\vspace*{-1cm}

\maketitle

\vspace*{-0.35cm}
\section*{Introduction}

The Brill--Noether theory of algebraic curves has its roots in the
19th century study of the projective embeddings induced by regular
functions on algebraic curves. The Brill--Noether number,
$\rho(g,r,d)=g-(r+1)(g-d+r)$ is the expected dimension of the space
$W^r_d(C)$, which parameterizes line bundles of type $\g{r}{d}$,
i.e. of degree $d$ and with at least $r+1$ linearly independent global
sections. In the 1970s and 1980s, Donagi, Eisenbud, Green, Harris, Lazarsfeld, Martens, Morrison,
Mumford,  Reid, Saint-Donat, and Tyurin \cite{DonagiMorrison_linsystemsonk3,eisenbud_harris,GreenLaz,harris,harris_mumford,Martens,Reid1976,Saint_Donat_Proj_Models_of_K3s,Tyurin_1987} studied the Brill--Noether
theory of curves on a K3 surface.  In the early 21st century, Mukai
introduced a Brill--Noether theory for polarized K3 surfaces, and
showed that in genus $g \leq 10$ and $12$, the Brill--Noether general
polarized K3 surfaces are precisely those admitting a so-called
``Mukai model'' in a projective homogeneous variety.  While much of
this work implicitly circles around the interplay between the
Brill--Noether theory of polarized K3 surfaces and the Brill--Noether
theory of its hyperplane section curves, the precise relationship
between these two notions, which is linked via the theory of
Lazarsfeld--Mukai bundles, is still somewhat mysterious.

A smooth projective curve $C$ admiting a $\g{r}{d}$ with $\rho(g,r,d)<0$ is called \emph{Brill--Noether special}. Following Mukai~\cite[Definition~3.8]{mukai:Fano_threefolds}, we say
that a polarized K3 surface $(S,H)$ of genus $g$ is \emph{Brill--Noether special} if there exists a nontrivial line bundle $L\neq H$ such that
\[
h^0(S,L)\, h^0(S,H-L)\ge h^0(S,H)=g+1.
\] 
If $(S,H)$ is Brill--Noether
special, then a smooth curve $C\in|H|$ admits a Brill--Noether special
line bundle, simply by restricting $L$ to $C$. The converse is an open
question, stated by
Knutsen~\cite[Remark~10.2]{Knutsen_k3_models_in_scrolls},\cite[Remark~2.2]{knutsen2013},
and attributed to Mukai.

\begin{questionintro}
	\label{Mukai Conj}
	Let $(S,H)$ be a polarized K3 surface, and $C \in |H|$ a smooth
	irreducible curve. If $C$ is Brill--Noether special, then is $(S,H)$
	is Brill--Noether special?
\end{questionintro}

If both $L$ and $H-L$ are globally generated, this is equivalent to
$\rho(g,r,d)<0$, where $L^2 = 2r - 2$ and $d=H.L$.  Indeed, by Greer,
Li, and Tian~\cite[Lemma~1.6]{greer-li-tian2014}, the Brill--Noether
special locus in the moduli space $\mathcal{K}_g$ of quasi-polarized K3
surfaces of genus $g$ is a union of Noether--Lefschetz loci, which are
parameterized by Brill--Noether special lattice markings, see
\Cref{Section BN spec k3s background}. 
	
For K3 surfaces admitting Mukai models, one can verify by hand that
\Cref{Mukai Conj} has a positive answer, cf.\
\cite[Theorem~1.3]{auel2020brillnoether},
\cite[Remark~1.4]{hoff_stagliano}.  A celebrated result of
Lazarsfeld~\cite{lazarsfeld:Brill-Noether_without_degenerations} shows
that if $H$ generates the Picard group of $S$, whence $(S,H)$ is
Brill--Noether general, then any smooth curve in the linear system of
$H$ is Brill--Noether general.  Hence \Cref{Mukai Conj} has a positive
answer for primitively polarized K3 surfaces of Picard rank $1$.
Lazarsfeld's result also holds for certain classes of polarized K3
surfaces of Picard rank 2, the so-called \emph{Knutsen K3
surfaces} of~\cite{ArapMarshburn_BN_general_curves_on_knutsen_k3s}. However, the
question remains open for general polarized K3 surfaces of Picard rank
$> 1$ and genus $> 12$.  In this work, we use new techniques to expand
the genera for which the question has a positive answer.

\begin{theoremintro}
	\Cref{Mukai Conj} has a positive answer for $g \le 19$.
\end{theoremintro}

Our approach is to take a Brill--Noether special $A\in\Pic(C)$ and
``lift'' it to line bundle $L$ on $S$, which should give a
Brill--Noether special marking on $\Pic(S)$. The proof would be
simplified if we could actually find $L$ such that $L\vert_C\cong
A$. However, this is not always the case. In general, we work with a
weaker notion of lift, a \emph{Donagi--Morrison lift}, whose existence
was conjectured by Donagi and Morrison, see \Cref{conj DM}. The
Donagi--Morrison conjecture, however, does not hold in general, and we
state a bounded version \Cref{conj Bounded Strong DM} that holds in
the range of genera we consider due to results of
Lelli-Chiesa~\cite{Lelli_Chiesa_2015} and our recent results~\cite{haburcak_2022}.

We remark that in genus $20$ and above, current techniques for lifting
line bundles are insufficient for \Cref{Mukai Conj}. In particular,
one would need to lift linear systems of rank $4$. Already in genus
$18$ and $19$, we encounter this issue, but can work around it by
finding a different linear system with rank $3$ that we can lift; this
trick, however, is no longer available in higher genus where Brill--Noether loci corresponding to linear systems of rank $4$ are \emph{expected maximal}, see \Cref{remark obtaining bp free g^3_d}.

Geometrically, we can think about \Cref{Mukai Conj} in the following
way. As in \cite{arbarello_bruno_sernesi}, we consider 
$$
\xymatrix@R=7pt@C=16pt{
 & \mathcal{P}_g \ar[rd]^(0.45){\pi} \ar[ld]_(0.45){\phi} & \\
\mathcal{M}_g & & \mathcal{K}_g
}
$$
where $\pi : \mathcal{P}_g \to \mathcal{K}_g$ is the universal smooth
hyperplane section, whose fiber above a polarized K3 surface $(S,H)$
of genus $g$ is the set of smooth and irreducible curves in the linear
system $|H|$ (it is an open substack of a $\P^g$-bundle over
$\mathcal{K}_g$), and $\phi : \mathcal{P}_g \to \mathcal{M}_g$ is the
forgetful map to the moduli space of smooth curves of genus~$g$.  In
$\mathcal{P}_g$, we want to compare the preimages of the
Brill--Noether special loci $\mathcal{K}_g^\text{BN} \subset
\mathcal{K}_g$ and $\mathcal{M}_g^\text{BN} \subset \mathcal{M}_g$.
We know that $\pi\inv \mathcal{K}_g^\text{BN} \subseteq \phi\inv
\mathcal{M}_g^\text{BN}$ and \Cref{Mukai Conj} asks whether the reverse containment also holds.

\subsection*{Outline}
In \Cref{Section BN spec k3s background}, we recall and clarify the
identification of Brill--Noether special polarized K3 surfaces via
lattice polarizations of their Picard group, following work of Greer,
Li, and Tian~\cite{greer-li-tian2014}. In \Cref{Section lifting line
bundles background}, we recall the Donagi--Morrison conjecture
\Cref{conj DM} on lifting Brill--Noether special line bundles. In
\Cref{Section Lazarsfeld--Mukai Bundles and Lifting and Generalized LM
Bundles}, we summarize how Lazarsfeld--Mukai bundles are used to prove
cases of \Cref{conj DM}. We also state a stronger conjecture on the
maximal destabilizing subsheaf of a Lazarsfeld--Mukai bundle
\Cref{conj Strong DM}, and a bounded version \Cref{conj Bounded Strong
DM} of the Donagi--Morrison conjecture, which holds in the genera we
consider. In \Cref{section DM lifts from quotients} we give useful numerical bounds on Donagi--Morrison
lifts obtained from Lazarsfeld--Mukai bundle techniques. We also show in \Cref{subsection strong DM implies Mukai} that \Cref{conj Strong DM} implies a positive answer to \Cref{Mukai Conj}, see \Cref{Theorem strong DM implies Mukai}. Finally, in
\Cref{Section BN special k3s}, we prove our main result \Cref{Theorem
bn special k3s in genus 14-19}.

\subsection*{Acknowledgments} 
We would like to thank Asher Auel, Andreas Knutsen, and Margherita
Lelli-Chiesa for helpful discussions as well as an explanation of the
history of \Cref{Mukai Conj}, including relaying personal
communications with Shigeru Mukai.  This work represents part of the
author's PhD thesis \cite{haburcak:thesis} at Dartmouth College, and is based upon work partially supported by the National Science Foundation under Grant No. DMS-2200845 and supported by a grant from the Simons Foundation (712097, to Asher Auel).

\section{Brill--Noether special K3 surfaces}\label{Section BN spec k3s background}

Throughout the paper, we let $(S,H)$ be a polarized K3 surface of genus $g\ge 2$, and $C\in|H|$ a smooth irreducible curve of genus $g$.

\begin{defn}[{\cite[Definition~3.8]{mukai:Fano_threefolds}}]
A polarized K3 surface $(S,H)$ of genus $g$ is \emph{Brill--Noether
special} if there exists a nontrivial line bundle $L\neq H$ such that
\[
h^0(S,L)\, h^0(S,H-L)\ge h^0(S,H)=g+1.
\] 
In this case, we call the sublattice of $\Pic(S)$ generated by $H$ and $L$, denoted by $\langle H, L \rangle$, a
\emph{Brill--Noether special marking} of $(S,H)$.
\end{defn}

A natural question is to determine the markings that a Brill--Noether
special polarized K3 surface admits. In the moduli space
$\mathcal{K}^\circ_g$ of polarized K3 surfaces of genus $g$, the
Noether--Lefschetz (NL) locus parameterizes K3 surfaces with Picard
rank $>1$. By Hodge theory, the NL locus is a union of countably many
irreducible divisors, which we call NL divisors. In
\cite{greer-li-tian2014}, Greer, Li, and Tian study the Picard group
of $\mathcal{K}^\circ_g$ via the NL divisors as well as the locus of
Brill--Noether special K3 surfaces in $\mathcal{K}^\circ_g$, which
they identify as a certain union of NL divisors. More generally, it is
convenient to work with the moduli space of primitively
quasi-polarized K3 surfaces, denoted $\mathcal{K}_g$ where
$\mathcal{K}_g\setminus \mathcal{K}^\circ_g$ is a divisor
parameterizing K3 surfaces containing a $(-2)$-exceptional curve. For
$g \ge 2$ and $d, r \ge 0$, we define the NL divisor
$\mathcal{K}^r_{g,d}$ to be the locus of polarized K3 surfaces
$(S,H)\in\mathcal{K}_g$ such that
\[
\Lambda^r_{g,d}=
\begin{array}{c|cc} \multicolumn{1}{c}{}
  & H    & L    \\\cline{2-3}
H & 2g-2 &d     \\
L & d    & 2r-2
\end{array}
\] 
admits a primitive embedding in $\Pic(S)$ preserving
$H$. Colloquially, we say that the line bundle $L$ is \emph{of type}
$\g{r}{d}$ on $(S,H)$.  We note that $\mathcal{K}^r_{g,d}$ is
irreducible by \cite{OGrady_irreducible_NL_divisors} whenever it is
nonempty.

For a general polarized K3 surface
$(S,H)\in\mathcal{K}^{r^\prime}_{g,d}$, curves $C\in|H|$ have an
induced line bundle $L\vert_C$ of interest, and one may ask about its
rank and degree. In particular, it may happen that
$L\vert_C$ is a $\g{r}{d}$, but $L$ is of type $\g{r^\prime}{d}$. In
this case, $r$ and $r^\prime$ are related by the following lemma.

\begin{lemma}[{\cite[Lemma 3.13]{haburcak_2022}}]\label{General 2r-2}
	Let $(S,H)$ be a polarized K3 surface of genus $g\ge2$, $C\in|H|$ be a smooth irreducible curve, and $L$ a globally generated line bundle on $S$ such that $L\vert_C$ is a $\g{r}{d}$ with $c_1(L).C=d<2g-2$. Then if $h^1(S,L)=0$, we have $L^2=2r-2-2h^1(S,L(-C))$.
\end{lemma}

In considering Brill--Noether special polarized K3 surfaces, we are naturally lead to consider two constraints. First, we have
$\rho(g,r,d)<0$ as the restriction of $L$ to a smooth curve $C \in
|H|$ is (usually) a Brill--Noether special $\g{r}{d}$, cf.\
\cite[Lemma~1.1]{haburcak_2022}.  We call the constraint
$\rho(g,r,d)<0$ the \emph{Brill--Noether constraint}.  Second,
the Hodge index theorem implies that the discriminant of
$\Lambda^r_{g,d}$ is negative, and we define
\[
\Delta(g,r,d)\colonequals\operatorname{disc}\left(\Lambda^g_{d,2r-2}\right)=4(g-1)(r-1)-d^2=4(g-1)(r-1)-(\gamma(g,r,d)+2r)^2,
\]
calling the constraint $\Delta(g,r,d)<0$ the \emph{Hodge constraint}.
We rephrase the characterization of the Brill--Noether special locus
in $\mathcal{K}_g$ in \cite{greer-li-tian2014} using the Hodge and
Brill--Noether constraints.

\begin{prop}[{\cite[Lemma 2.8]{greer-li-tian2014}}]
\label{prop:GLT}
The Brill--Noether special locus in $\mathcal{K}_g$ is a union of the
NL divisors $\mathcal{K}^r_{g,d}$ satisfying $0 \le d \le g-1$,
$\rho(g,r,d)<0$, and $\Delta(g,r,d)<0$.
\end{prop}
\begin{proof}
By replacing $L$ with $H-L$, we can always assume that $d\le g-1$.  In
\cite[Lemma 2.8]{greer-li-tian2014}, the lower bound
$\sqrt{4(g-1)(r-1)}<d$ is simply $\Delta(g,r,d)<0$ while the other
upper bound $d\le r+g-\frac{g+1}{r+1}$ is equivalent to $\rho(g,r,d)\le-1$.
\end{proof}
	
We remark that the only NL loci that are Brill--Noether special with
$d \leq 1$ is the locus $\mathcal{K}^1_{g,1}$ of elliptic K3s with a
section, but on this locus $H$ has a fixed component.

\section{Lifting Brill--Noether special line bundles}
\label{Section lifting line bundles background}

We collect useful facts on lifting Brill--Noether special line bundles from a curve $C$ to line bundles on the K3 surface $S$. Clearly, one cannot always lift a line bundle from $C$ to $S$, however, one expects to be able to lift Brill--Noether special line bundles in a weaker sense.

The \emph{Clifford index} of a line bundle $A$ on a smooth projective
curve $C$ is the integer $\gamma(A)=\deg(A)-2\, r(A)$ where $r(A)
= h^0(C,A)-1$ is the rank of $A$. The \emph{Clifford index of $C$} is
the integer
\[
\gamma(C) \colonequals \min \{\gamma(A) ~\vert~ h^0(C,A) \geq
2\text{ and }h^1(C,A) \geq 2 \},
\]
and we say that a line bundle $A$ on $C$ \emph{computes} the Clifford
index of $C$ if $\gamma(A)=\gamma(C)$. Clifford's theorem states that
$0\le\gamma(C)\le\lfloor\frac{g-1}{2}\rfloor$, and furthermore that
$\gamma(C)=\lfloor\frac{g-1}{2}\rfloor$ when $C$ is a general curve of genus ~$g$.

\begin{defn}[{\cite[Definition 1.3]{haburcak_2022}}]
Let $A$ be a Brill--Noether special $\g{r}{d}$ on a curve $C$ of
genus~$g$, i.e. $\rho(g,r,d)<0$. We say $A$ is \emph{non-computing} if
$\gamma(r,d)>\lfloor \frac{g-1}{2} \rfloor$, that is, $A$ is a
Brill--Noether special $\g{r}{d}$ that cannot compute the Clifford
index of $C$. We note that non-computing $\g{r}{d}$ only appear in genus
$\ge 14$.
\end{defn}

Next, we introduce some notions required for the statement of the
Donagi--Morrison conjecture.

\begin{defn}\label{definition linear system contained in restriction of linear system}
	Let $S$ be a K3 surface, $C\subset S$ be a curve, and $A\in\Pic(C)$ and $M\in\Pic(S)$ be line bundles. We say that the linear system $|A|$ is \emph{contained in the restriction} of $|M|$ to $C$ when for every $D_0\in|A|$, there is some divisor $M_0\in|M|$ such that $D_0\subset C\cap M_0$.
\end{defn}	
\begin{defn}\label{definition adapted}
A line bundle $M$ is \emph{adapted} to $|H|$ when 
\begin{enumerate}[label=(\roman*)]
\item $h^0(S,M)\ge 2$ and $h^0(S,H - M)\ge2$; and
\item $h^0(C,M|_C)$ is independent of the smooth curve $C\in|H|$.
	\end{enumerate}  
\end{defn}
Thus whenever $M$ is adapted to $|H|$, condition (i) ensures that
$M|_C$ contributes to $\gamma(C)$, and condition (ii) ensures that
$\gamma(M|_C)$ is constant as $C$ varies in its linear system and is
satisfied if either $h^1(S,M)=0$ or $h^1(S,H - M)=0$.

\begin{conj}[Donagi--Morrison Conjecture, \cite{Lelli_Chiesa_2015} Conjecture 1.3]\label{conj DM}
Let $(S,H)$ be a polarized K3 surface and $C\in|H|$ be a smooth
irreducible curve of genus $\ge 2$. Suppose $A$ is a complete
basepoint free $\g{r}{d}$ on $C$ such that $d\le g-1$ and
$\rho(g,r,d)<0$. Then there exists a line bundle $M\in\Pic(S)$ adapted
to $|H|$ such that $|A|$ is contained in the restriction of $|M|$ to
$C$ and $\gamma(M|_C)\le \gamma(A)$.
\end{conj}

\Cref{conj DM} was verified for $r=1$ by Donagi and Morrison
\cite{DonagiMorrison_linsystemsonk3}, and ``bounded versions'' exist for
$r=2$ \cite{Lelli_Chiesa_2013} and $r=3$ \cite{haburcak_2022}, where
$d$ must be bounded above by some constant depending on $C$ and $S$. Some
further hypotheses are required, as shown in \cite[Appendix
A]{Lelli_Chiesa_2015}, where Knutsen and Lelli-Chiesa construct
counterexamples.

\begin{defn}
In the statement of \Cref{conj DM}, we call such a line bundle $M$ on $S$ a \emph{Donagi--Morrison
lift} of $A$.
We call a line bundle $M$ on $S$ a \emph{potential Donagi--Morrison
lift} of $A$ if $M$ satisfies $\gamma(M|_C)\le\gamma(A)$ and
$d(M|_C)\ge d(A)$. Note that a Donagi--Morrison lift is a potential
Donagi--Morrison lift.
\end{defn}


We will postpone a discussion of a modified \Cref{conj DM} to
\Cref{Section Lazarsfeld--Mukai Bundles and Lifting and Generalized LM
Bundles}, after we recall how Lazarsfeld--Mukai bundles are used to
construct Donagi--Morrison lifts.

Sometimes, one can do better than a Donagi--Morrison lift, and find a
line bundle $L$ on $S$ that is a \emph{lift} of $A$, i.e., such that
$L|_C = A$.  Under some technical hypotheses, Lelli-Chiesa finds lifts
of Brill--Noether special line bundles that compute the Clifford
index.

\begin{theorem}[{\cite[Theorem 4.2]{Lelli_Chiesa_2015}}]\label{LC Thm 4.2}
Let $A$ be a complete $\g{r}{d}$ on a non-hyperelliptic and
non-trigonal curve $C\subset S$ such that $d\le g-1$, $\rho(A)<0$, and
$\gamma(A)=\gamma(C)$. Assume $H=\OO_S(C)$ is ample and there
is no irreducible elliptic curve $\Sigma\subset S$ such that the
following holds: \begin{enumerate}
		\item[{\normalfont($\ast$)}] $\Sigma.C=4$ and no irreducible genus $2$ curve $B\subset S$ such that $B.C=6$.
\end{enumerate} Then, one of the following occurs: 
\begin{enumerate}[label={\normalfont(\roman*)}]
\item There exists a line bundle $L\in\Pic(S)$ adapted to $|H|$ such
that $A\cong L|_C$.

\item The line bundle $A$ is of type $\g{1}{d}$, in which case \Cref{conj DM} holds for $A$.
\end{enumerate}
\end{theorem}

This lifting result vastly strengthens that of Green and
Lazarsfeld~\cite{GreenLaz}, which itself resolved a conjecture of
Harris and Mumford, as modified by Green~\cite[Conjecture~5.8]{green}.
Harris and Mumford~\cite{harris_mumford} conjectured that the gonality
remains constant for smooth curves in an ample linear system on a K3
surface, but a counterexample was constructed by Donagi and
Morrison~\cite{DonagiMorrison_linsystemsonk3}, though this turned out
to be the unique counterexample, see \cite{ciliberto_pareschi},
\cite{knutsen:two_conjectures}.  Green modified the conjecture to a
statement about the constancy of the Clifford index of smooth curves
in an ample linear system on a K3.  See \cite[\S2]{aprodu} for further
details.

\section{Lazarsfeld--Mukai bundles and lifting}
\label{Section Lazarsfeld--Mukai Bundles and Lifting and Generalized LM Bundles}

In this section, we give additional numerical constraints on
Donagi--Morrison lifts obtained as determinants of stable quotients.
First, we briefly recall some facts about Lazarsfeld--Mukai bundles
(LM bundles). 
Let $\iota\colon C \hookrightarrow S$ be a smooth irreducible curve of
genus $g$ in $S$ and $A$ a basepoint free line bundle on $C$ of type
$\g{r}{d}$. We define a vector bundle $F_{C,A}$ on $S$ via the short
exact sequence
\[
0 \longrightarrow F_{C,A} \longrightarrow H^0(C,A) \otimes \OO_X
\xrightarrow{\;ev\;} \iota_* A \longrightarrow 0
\]
The LM bundle $E_{C,A}\colonequals F_{C,A}^\vee$ associated to $A$ on
$C$ then sits in the short exact sequence
\[
0 \longrightarrow H^0(C,A)^\vee \otimes\OO_S \longrightarrow E_{C,A} \otimes \OO_X
\xrightarrow{\;ev\;} \iota_\ast(\omega_C\otimes A^\vee) \longrightarrow 0
\]
from which the following facts are readily proved, cf.\
\cite[Proposition~1.2]{aprodu}.

\begin{prop}\label{LM bundle props}
	Let $E_{C,A}$ be a LM bundle associated to a basepoint free line bundle $A$ of type $\g{r}{d}$ on $C\subset S$, then:
	\begin{itemize}
		\item $\det E_{C,A}=c_1(E_{C,A})=[C]$ and $c_2(E_{C,A})=\operatorname{deg}(A)$;
		\item $\operatorname{rk}(E_{C,A})=r+1$ and $E_{C,A}$ is globally generated off the base locus of $\iota_\ast(\omega_C\otimes A^\vee)$;
		\item $h^0(S,E_{C,A})=h^0(C,A)+h^0(C,\omega_C\otimes A^\vee)=2r+1+g-d=g-(d-2r)+1$;
		\item $h^1(S,E_{C,A})=h^2(S,E_{C,A})=0$, $h^0(S,E^\vee_{C,A})=h^1(S,E^\vee_{C,A})=0$;
		\item $\chi(F_{C,A}\otimes E_{C,A})=2(1-\rho(g,r,d))$.
	\end{itemize}
\end{prop}

In lifting line bundles on a curve $C\in|H|$ to line bundles on a
polarized K3 surface $(S,H)$, it is also useful to work with
generalized Lazarsfeld--Mukai bundles, which are defined in
\cite{Lelli_Chiesa_2015}.

\begin{defn}[{\cite[Definition 1]{Lelli_Chiesa_2015}}]\label{gLM bundle defn}
	A torsion free coherent sheaf $E$ on $S$ with $h^2(S,E)=0$ is called a \emph{generalized Lazarsfeld--Mukai bundle} (gLM bundle) of type (I) or (II), respectively, if
	\begin{enumerate}[label=(\Roman*)]
		\item $E$ is locally free and generated by global sections off a finite set;\\
		or
		\item $E$ is globally generated.
	\end{enumerate}
\end{defn}

A LM bundle is a gLM bundle of type (I), and a partial converse is
given in \cite[Proposition~1.3]{aprodu}, with special case as follows.
 
\begin{remark}[{\cite[Remark 1]{Lelli_Chiesa_2015}}]\label{remark glm is lm bundle}
If $E$ is a gLM bundle of type (I) and (II) then $E$ is the LM bundle
associated with a base point free linear series $V \subseteq H^0(C,A)$
on a smooth irreducible curve $C\subset S$ such that $\omega_C \otimes
A^\vee$ is also basepoint free.  Specifically, $E=E_{C,(A,V)}$ is the
dual of the kernel of the evaluation map $V\otimes \OO_S\to
A$. Furthermore, $V=H^0(C,A)$ if and only if $h^1(S,E)=0$, in which
case $E$ is just the LM bundle associated to $A$.
\end{remark}

\begin{defn}
A line bundle $A$ on a smooth curve $C$ is called
\emph{primitive} if both $A$ and $\omega_C\otimes A^\vee$ are
basepoint free.
\end{defn}

The definition of the Clifford invariant of a linear
system on a curve motivates the following.

\begin{defn}
	Let $E$ be a gLM bundle. The \emph{Clifford index of $E$} is:
	\[\gamma(E)\colonequals c_2(E) - 2(\operatorname{rk}(E)-1).\]
\end{defn}

Indeed, for a LM bundle $E_{C,A}$ associated to a line bundle $A$ of
type $\g{r}{d}$ on a smooth curve $C\subset S$, one has $\gamma(E_{C,A})=\gamma(A)$ by \Cref{LM bundle props}.

\smallskip

Next, we recall some useful results using Lazarsfeld--Mukai bundles to
find lifts of Brill--Noether special line bundles. In particular, for
low $r$, we can add a bound that forces sub-line bundles to be
non-trivial.

\begin{prop}[{\cite[Proposition 3.17]{haburcak_2022}}]\label{Prop Proof Strategy}
	Let $(S,H)$ be a polarized K3 surface and $A$ be a complete
        basepoint free $g^r_d$ on a smooth irreducible curve $C\in|H|$
        with $r\ge 2$ and let $E=E_{C,A}$. Suppose that $E$ sits in a
        short exact sequence \[\xymatrix{0\ar[r] & N\ar [r]& E \ar[r]
        & E/N\ar[r] & 0}\] for some line bundle $N$ and
        $c_2(E)=d<\frac{g(r-1)}{r}+\frac{2g-2}{r(r+1)}+r-\frac{1}{r}$. If
        $E/N$ is stable, or $E/N$ is semistable and there are no
        elliptic curves on $S$, then $|A|$ is contained in the
        restriction to $C$ of the linear system $|H - N|$ on
        $S$. Moreover, $H - N$ is adapted to $|H|$ and $\gamma\bigl((H - N)|_C\bigr)\le d-r-3 $.
\end{prop}

The condition that $E_{C,A}/N$ be stable is made to ensure that
$\gamma\bigl((\det E_{C,A}/N)|_C\bigr)$ does not depend on the curve
$C\in|H|$, for details see the proof of \cite[Proposition
3.17]{haburcak_2022}. Following her proof of \Cref{conj DM} when
$\gamma(A)=\gamma(C)$, Lelli-Chiesa reduced the condition that
$E_{C,A}/N$ be stable to simply finding a saturated line bundle
$N\hookrightarrow E_{C,A}$, \cite[Proposition 5.1]{Lelli_Chiesa_2015}.

\begin{prop}[{\cite[Proposition 5.1]{Lelli_Chiesa_2015}}]\label{LC prop 5.1}
	Let $A$ be primitive and assume the existence of a globally generated line bundle $N\subseteq E_{C,A}$ which is saturated. Then \Cref{conj DM} holds with $M\colonequals \det(E_{C,A}/N)=H-N$ if $h^1(S,N)=0$, and with $M\colonequals H-\Sigma$ if $c_1(N)=k\Sigma$ for an irreducible elliptic curve $\Sigma\subset S$ and integer $k\ge 2$.
\end{prop}

The proofs of cases of \Cref{conj DM} have all used this idea. However, finding such a line bundle is generally done by exhibiting $N$ as the maximal destabilizing subsheaf of $E_{C,A}$, as this allows one to use stability arguments and use the structure of $E_{C,A}$. Thus one can ask whether a stronger version of \Cref{conj DM} holds.

\begin{conj}[Strong Donagi--Morrison Conjecture]\label{conj Strong DM}
	Let $(S,H)$ be a polarized K3 surface of genus $g$ and $A$ be a complete basepoint free $g^r_d$ on a smooth irreducible curve $C\in|H|$ with $r\ge 2$ and $\rho(g,r,d)<0$. Then there is a nontrivial line bundle $N\hookrightarrow E_{C,A}$ with $h^0(S,N)\ge 2$ such that $E_{C,A}/N$ is stable.
\end{conj}

As stated, this conjecture is false, see \cite[Appendix
A]{Lelli_Chiesa_2015} and \cite[Section 6.7]{haburcak_2022}. However,
a bounded version as in
\cite{haburcak_2022,Lelli_Chiesa_2013,Lelli_Chiesa_2015,Reid1976} may
be reasonable.

\begin{conj}[Bounded Strong Donagi--Morrison Conjecture]\label{conj Bounded Strong DM}
	Let $(S,H)$ be a polarized K3 surface of genus $g$ and $A$ be a complete basepoint free $g^r_d$ on a smooth irreducible curve $C\in|H|$ with $r\ge 2$ and $\rho(g,r,d)<0$. Then there is a bound $\beta$ depending on $C$ and $S$ such that if $d<\beta$, then there is a line bundle $N\hookrightarrow E_{C,A}$ with $h^0(S,N)\ge 2$ such that $E_{C,A}/N$ is stable.
\end{conj}

This conjecture is verified when $r=2$ in \cite{Lelli_Chiesa_2013} and when $r=3$ in \cite{haburcak_2022}. For completeness, we recall the bounds in the $r=3$ case.

\begin{theorem}[{\cite[Theorem 5.1]{haburcak_2022}}]\label{theorem lifting g3ds general}
	Let $(S,H)$ be a polarized K3 surface of genus $g\neq
        2,3,4,8$, $C\in |H|$ a smooth irreducible curve of Clifford
        index $\gamma$, and $A$ a $\g{3}{d}$ on
        $C$. Let $$m\colonequals \min\{D^2 ~\vert~ D\in\Pic(S),~
        D^2\ge 0, \text{ $D$ is effective}\}$$ (i.e. there are no
        curves of genus $g^\prime<\frac{m+2}{2}$ on $S$),
        and $$\mu\colonequals\min\{\mu(D) ~\vert~ D\in\Pic(S),~
        D^2\ge0,~ \mu(D)>0\}.$$
        If $$d<\min\left\{\frac{5}{4}\gamma+\frac{\mu+m+9}{2},~\frac{5}{4}\gamma+\frac{m}{2}+5,~
        \frac{3}{2}\gamma+5,~
        \frac{\gamma+g-1}{2}+4\right\},$$ then \Cref{conj
        Strong DM} holds for $A$. In particular, there is a line
        bundle $M\in\Pic(S)$ adapted to $|H|$ such that $|A|$ is
        contained in the restriction of $|M|$ to $C$ and $\gamma(M|_C)\le\gamma(A)$. Moreover, one has $c_1(M).C\le\frac{3g-3}{2}$.
\end{theorem}

We now collect some useful facts about quotients of Lazarsfeld--Mukai bundles.

\begin{lemma}[{\cite[Lemma 3.8]{haburcak_2022}}]\label{coker is gLM}
	Let $N\in \Pic(S)$ be nontrivial and globally generated with
        $h^0(S,N) \ne 0$. Let $E=E_{C,A}$ and suppose we have a short exact sequence \[\xymatrix{0\ar[r] & N \ar[r] & E \ar[r] & E/N \ar[r] & 0}\] with $E/N$ torsion free. Then $E/N$ satisfies $h^1(S,E/N)=h^2(S,E/N)=0$. If $A$ is primitive, then $E/N$ is a gLM bundle of type (II). If we further assume that $E/N$ is locally free, then it is a LM bundle for a smooth irreducible curve $D\in|H-N|$. If $A$ is not primitive and $E/N$ is assumed locally free, then $E/N$ is a gLM bundle of type (I). In any of the above cases, we have
	\begin{itemize}
		\item $M\colonequals c_1(E/N)=H-N$;
		\item $c_2(E/N)=d+N^2-H. N=d+M^2-H.M$;
		\item $\gamma(E/N)=\gamma(E_{C,A})+M^2 - H. M+2$.
	\end{itemize}
\end{lemma}

\begin{remark}
	If $A$ is of type $\g{r}{d}$ and $M=H-N$ is a lift of $A$ with $M^2=2r-2$, then the last equality gives $\gamma(E/N)=\gamma(A)+(2r-2)-d+2=0$.
\end{remark}
\begin{remark}\label{bounded c_2}
	The proof of \Cref{coker is gLM} also shows that if $A$ is primitive and $N\subset E=E_{C,A}$ is any subsheaf (not necessarily a line bundle) such that $E/N$ is torsion free (e.g. obtained through a Harder--Narasimhan filtration), then $E/N$ is a gLM bundle of type (II). Moreover, by \cite[Proposition 2.7]{Lelli_Chiesa_2015}, if $c_1(E/N)^2=0$, then $c_2(E/N)=0$.
\end{remark}

\section{Donagi--Morrison lifts obtained from quotients}\label{section DM lifts from quotients}

The Donagi--Morrison lift of a $\g{r}{d}$ may have a different type, as in \Cref{General 2r-2}. However, if the lift is obtained as the determinant of a quotient of a LM bundle $E_{C,A}$ as in \cite{haburcak_2022,Lelli_Chiesa_2013,Lelli_Chiesa_2015}, there are additional numerical restrictions that hold. Namely, if $M$ is a Donagi--Morrison lift of a $\g{r}{d}$ of type $\g{r^\prime}{d^\prime}$, then $r^\prime$ is constrained by $r$ and the Clifford index of the quotient of $E_{C,A}$. Throughout this section, we restrict ourselves to the case of \Cref{conj Strong DM}. That is, $A$ is a complete basepoint free Brill--Noether special $\g{r}{d}$ on $C$ and there is a short exact sequence \[0\to N\to E_{C,A} \to E \to 0\] where $N$ is a nontrivial line bundle with $h^0(S,N)\ge 2$.

\begin{remark}
	We briefly recall some facts about the moduli space of stable and semistable sheaves on K3 surfaces from \cite{huybrechts_2016}. Let $E$ be a coherent sheaf, and \[v(E)=ch(E)\sqrt{\operatorname{td}(S)}=(\operatorname{rk}(E), c_1(E), \chi(E)-\operatorname{rk}(E))=(\operatorname{rk}(E),c_1(E), ch_2(E)+\operatorname{rk}(E))\] be its Mukai vector, where $\operatorname{td}(S)$ is the Todd class of $S$ and $ch_2(E)=\frac{c_1(E)^2-2c_2(E)}{2}$ is the part of the Chern character of $E$ in $H^4(S,\mathbb{Z})$. The moduli space of stable coherent sheaves with Mukai vector $v(E)$ is empty or has dimension $2+\langle v(E), v(E)\rangle$. Thus when $E$ is a stable coherent sheaf \[\langle v(E), v(E)\rangle=\left(1-\operatorname{rk}(E)\right)c_1(E)^2+2\operatorname{rk}(E)c_2(E)-2\operatorname{rk}(E)^2\ge -2 .\]
\end{remark}

\begin{prop}\label{prop dm lift restrictions if quotient is stable}
	Let $M=\det(E)$ be of type $\g{r^\prime}{d^\prime}$, and suppose that $E$ is stable. Then
	\begin{enumerate}[label=\normalfont(\roman*)]
		\item $\gamma(M)\colonequals d^\prime-2r^\prime\le \gamma(A)+r-r^\prime+\frac{r^\prime}{r}-1$, and
		\item $c_2(E)\ge \frac{(\operatorname{rk}(E)-1)c_1(E)^2}{2\operatorname{rk}(E)}+\operatorname{rk}(E)-\frac{1}{\operatorname{rk}(E)}$.
	\end{enumerate}
\end{prop}
\begin{proof}
	As $E$ is stable, \[\langle v(E), v(E)\rangle=(1-\operatorname{rk}(E))c_1(E)^2+2\operatorname{rk}(E)c_2(E)-2\operatorname{rk}(E)^2\ge -2.\] Noting that $c_2(E)=c_2(E_{C,A})-d^\prime+2r^\prime-2$, we compute 
	\begin{align*}
		&(1-r)(2r^\prime -2)+(2r)(d-d^\prime+2r^\prime -2)-2r^2\ge -2\\
		\iff& r^\prime -rr^\prime -r+rd-rd^\prime +2rr^\prime -r^2 \ge 0\\
		\iff& rr^\prime+r-rd+rd^\prime -2rr^\prime +r^2 \le r^\prime\\
		\iff& d^\prime -r^\prime -d+r\le\frac{r^\prime}{r}-1\\
		\iff& \gamma(M)-\gamma(A) \le r-r^\prime +\frac{r^\prime}{r}-1,
	\end{align*}
from which (i) follows.

Finally, (ii) follows by rearranging $\langle v(E),v(E)\rangle\ge -2$.
\end{proof}
\begin{remark}\label{remark dm lift restrictions if quotient is stable}
	If $\operatorname{rk}(E)=3$ and $c_1(E)^2\ge 2$, then $c_2(E)\ge 4$. Since $c_2(E)=c_2(E_{C,A})-\gamma(M)-2$, if $c_2(E)\ge 3$ then $\gamma(M)\le c_2(E_{C,A})-6$, which is the Clifford index of a $\g{3}{d}$. This observation will be useful later when we consider lifts of non-primitive $\g{3}{d}$s.
\end{remark}

\subsection{Brill--Noether special K3 surfaces from stable quotients}\label{subsection strong DM implies Mukai}

We show that \Cref{conj Strong DM} implies a positive answer for \Cref{Mukai Conj}. While \Cref{conj Strong DM} is false, this still shows that whenever the conclusion holds, in particular when \Cref{conj Bounded Strong DM} holds then \Cref{Mukai Conj} has a positive answer. We show that under these assumption that \Cref{conj Strong DM} holds, $\langle H, \det(E) \rangle$ is a Brill--Noether special marking on $\Pic(S)$. Throughout, we let $M\colonequals\det (E)$ be of type $\g{r^\prime}{d^\prime}$, and we write $\gamma(M)=d^\prime - 2r^\prime$.

We first reduce to the case that $\gamma(A)>\gamma(C)=\left\lfloor \frac{g-1}{2} \right\rfloor$. Arguing by the Clifford index allows us to freely assume that $A$ is complete and basepoint free, even that $A$ is primitive (though this is not required), as we only require $A$ to becomplete and basepoint free in order to use \Cref{Prop Proof Strategy} to see that $M$ is a Donagi--Morrison lift of $A$. When $\gamma(A)\le\left\lfloor \frac{g-1}{2}\right\rfloor$, the lifting results of Lelli-Chiesa \cite[Theorem 4.2]{Lelli_Chiesa_2015} and Knutsen \cite[Lemma 8.3]{Knutsen2001} give a Brill--Noether special marking on $\Pic(S)$. We note that this completely deals with the case that $A$ is of type $\g{1}{d}$. Thus we assume that $r\ge 2$ and $\gamma(A)>\gamma(C)=\left\lfloor \frac{g-1}{2} \right\rfloor$ for the remainder of this section. We begin with more constraints on the Donagi--Morrison lifts coming from stable quotients of $E_{C,A}$. 

\begin{lemma}
	We have $r^\prime \ge r$.
\end{lemma}
\begin{proof}
	From \Cref{prop dm lift restrictions if quotient is stable} (i) and the fact that $c_2(E)=c_2(E_{C,A})-\gamma(M)-2$, we have \[\gamma(M)\le \min \left\{\gamma(A)+r-r^\prime -1+ \left\lfloor \frac{r^\prime}{r} \right\rfloor,~ \gamma(A)+r-2+\left\lfloor\frac{1}{r}\right\rfloor \right\}.\] Since $r\ge 2$, we have \[\gamma(M)\le \min \left\{\gamma(A)+r-r^\prime -1+ \left\lfloor \frac{r^\prime}{r} \right\rfloor,~ \gamma(A)+r-2\right\}.\]
	
	Suppose for contradiction that $r^\prime< r$, and write $r^\prime = r-\delta$ for some positive integer $\delta$. Thus $\gamma(M)\le \gamma(A)+\delta-1$. Moreover, since $M$ is a Donagi--Morrison lift of $A$ we have $d\le d^\prime$, thus \[\gamma(A)+2\delta = d-2r+2\delta = d-2r^\prime \le d^\prime - 2r^\prime = \gamma(M).\] Therefore $\gamma(A)+2\delta \le \gamma(M) \le \gamma(A)+\delta-1$, and thus $\delta\le -1$, which is a contradiction. Thus $r^\prime \ge r$, as desired.
\end{proof}

\begin{remark}
	We write $r^\prime=r+\delta$ for the remainder of this section.
\end{remark}

\begin{lemma}
	We have $\gamma(M)\le \gamma(A)$.
\end{lemma}
\begin{proof}
	We have $\gamma(M)\le \gamma(A)+r-r^\prime -1+\left\lfloor \frac{r^\prime}{r}\right\rfloor$, from which it follows that $\gamma(M)\le \gamma(A)-\delta + \left\lfloor\frac{\delta}{r}\right\rfloor$. Since $r\ge 2$ and $\delta\ge 0$, clearly $\left\lfloor\frac{\delta}{r}\right\rfloor \le \delta$, from which the result follows.
\end{proof}

We summarize useful inequalities and formulae for Donagi--Morrison lifts of complete basepoint free line linear systems coming from \Cref{conj Strong DM}.
\begin{align}
	\gamma(M)&\le \gamma(A)-\delta+\left\lfloor\frac{\delta}{r}\right\rfloor \label{eqn_gammaM<gammaA+del} \\
	d^\prime &\le d+\delta + \floor{\frac{\delta}{r}} \label{eqn_dprime<d+del}\\
	\rho(g,r^\prime, d^\prime) &= \rho(g,r,d)+\delta\left( d^\prime-2r-\delta -g-1\right)
\end{align}
Substituting \Cref{eqn_dprime<d+del} gives 
\begin{align}
	\rho(g,r^\prime, d^\prime)\le \rho(g,r,d)+\delta\left( \gamma(A)-g-1+\floor{\frac{\delta}{r}} \right). \label{eqn_rho(g,rprime, dprime)<rho+del}
\end{align}

\begin{prop}
	We have $\gamma(A)-g-1+\floor{\frac{\delta}{r}}<0$.
\end{prop}
\begin{proof}
	We may assume that $\gamma(M)\ge \floor{\frac{g-1}{2}}> \frac{g-1}{2}-1$, otherwise $\gamma(M\vert_C)<\floor{\frac{g-1}{2}}$ contradicting the assumption that $\gamma(C)=\floor{\frac{g-1}{2}}$. Thus we have 
	\begin{align*}
		&2k \le d^\prime -2r +\frac{3}{2} -\frac{g}{2}\\
		\Cref{eqn_dprime<d+del} &\implies 2k\le d+\delta +\floor{\frac{\delta}{r}}-2r+\frac{3}{2}-\frac{g}{2}\\
		&\implies \delta-\floor{\frac{\delta}{r}} \le \gamma(A) +\frac{3}{2} -\frac{g}{2}.
	\end{align*}
	
	Since $-\frac{\delta}{r}\le -\floor{\frac{\delta}{r}}$, we have \[\delta\left(\frac{r-1}{r}\right)\le \gamma(A)+\frac{3}{2} - \frac{g}{2}.\] As $r\ge 2$, we have $\frac{r-1}{r}\ge 0$, hence 
	\begin{align}
		\delta &\le \left(\frac{r}{r-1}\right)\left(\gamma(A)+\frac{3}{2}-\frac{g}{2}\right) \label{eqn_del<bla}\\
		\implies \floor{\frac{\delta}{r}}&\le \left(\frac{1}{r-1}\right)\left(\gamma(A)+\frac{3}{2}-\frac{g}{2}\right) \label{eqn_del/r<bla}.
	\end{align}
	
	The inequality $\rho(g,r,d)<0$ is equivalent to $\gamma(A)<\frac{rg}{r+1}-r$. Thus substituting \Cref{eqn_del/r<bla}
	\begin{align*}
		\gamma(A)-g-1+\floor{\frac{\delta}{r}} &< \frac{rg}{r+1}-r-g-1+\left(\frac{1}{r-1}\right)\left(\frac{rg}{r+1}-r+\frac{3}{2}-\frac{g}{2}\right) \\
		&= \frac{-rg+3g-2r^3-4r^2+3r+5}{2(r+1)(r-1)},
	\end{align*} 
	which is negative for $r\ge 3$. 
	
	It remains to show the desired inequality when $r=2$. In this case, \Cref{eqn_del<bla} implies that $\delta\le \frac{g}{3}-1$. Thus 
	\begin{align*}
		\gamma(A)-g-1+\floor{\frac{\delta}{r}}&\le \frac{2g}{3}-2-g-1+\frac{g}{3r}-\frac{1}{r}\\
		&= \frac{g(1-r)}{3r}-3-\frac{1}{r},
	\end{align*}
	which is negative as $r\ge 1$, as desired.
\end{proof}

\begin{cor}
	$\rho(g,r^\prime, d^\prime)\le\rho(g,r,d)$
\end{cor}

In total, we have proved that the Donagi--Morrison lift of $A$ obtained as $\det(E)$ induces a marking $\Lambda^{r^\prime}_{g,d^\prime}$ with $\rho(g,r^\prime,d^\prime)<0$. 

\begin{theorem} \label{Theorem strong DM implies Mukai}
	\Cref{conj Strong DM} gives a positive answer to \Cref{Mukai Conj}.
\end{theorem}
\begin{proof}
	Since $h^0(S,N)\ge 2$, we have $h^2(S,N)=0$. Likewise, as $d^\prime>0$, we have $h^2(S,M)=0$. For a line bundle $L\in\Pic(S)$, $\chi(S,L)=2+\frac{L^2}{2}$. Thus $h^0(S,N)\ge g-d^\prime +r^\prime$ and $h^0(S,M)\ge r^\prime +1$. Finally $\rho(g,r^\prime,d^\prime)<0$ shows that $\langle H, M \rangle$ is a Brill--Noether special marking on $\Pic(S)$. 
\end{proof}

\subsection{Donagi--Morrison lifts from generalized Lazarsfeld--Mukai bundle quotients}\label{subsection Dm lifts from gLM quotients}

There are additional constraints when $E$ is a (g)LM bundle. Note that $E$ is a gLM bundle of type (II) if $A$ is primitive. Throughout, we let $M\colonequals\det (E)$ be of type $\g{r^\prime}{d^\prime}$, and we write $\gamma(M)=d^\prime - 2r^\prime$.

\begin{prop}\label{prop lm bundle quotient stability}
	Suppose $E$ is a LM bundle with $\gamma(E)=k$. If $r^\prime > r+\frac{rk}{r-1}$, then $E$ is not stable.
\end{prop}
\begin{proof}
	$E=E_{D,B}$ is a LM bundle for a smooth irreducible curve $D\in|M|$ of genus $r^\prime$ and a line bundle $B\in\Pic(D)$. Since $\gamma(E)=k$ and $\operatorname{rk}(E)=r$, $B$ is a $\g{r-1}{k+2r-2}$ on $D$. We compute \[\rho(g(D), r(B), d(B))=\rho(r^\prime, r-1, k+2r-2),\] and see that $\rho(g(D), r(B), d(B))<0$ if and only if $r^\prime > r+\frac{rk}{r-1}$, hence $E$ is not stable.
\end{proof}
\begin{cor}\label{cor computing spec gamma lifts}
	If the LM bundle $E$ above is stable with $\gamma(E)=0$ and $\gamma(M\vert_C)=\gamma(A)$, then $M\vert_C\cong A$.
\end{cor}
\begin{proof}
	$M$ is a Donagi--Morrison lift of $A$. Since $E$ is stable, \Cref{prop lm bundle quotient stability} implies we have $r^\prime=r$, hence as $H.M-M^2+2=\gamma(A)-\gamma(E)$, we have $\operatorname{deg}(M\vert_C)=\operatorname{deg}(A)$. Therefore $\gamma(M\vert_C)=\gamma(A)$ implies that $h^0(C,M\vert_C)=r+1$. Thus $M\vert_C \cong A$, as desired.
\end{proof}
\begin{remark}\label{rk quotient not LM bundle}
	If $E$ above is not a LM bundle and only a gLM bundle, then $E^{\vee\vee}$ is a LM bundle of Clifford index $\gamma(E^{\vee\vee})= \gamma(E)-\ell(\kappa)$, where $\kappa= E^{\vee\vee}/E$ is the $0$-dimensional sheaf where $E$ is not locally free. Moreover, $E^{\vee\vee}$ is a LM bundle for a $\g{r-1}{k+2r-2-\ell(\kappa)}$ on a smooth irreducible curve $D\in |M|$. Repeating the same calculation shows that if $r^\prime>r+\frac{r(k-\ell(\kappa))}{r-1}$, then $E^{\vee\vee}$ is not stable.
\end{remark}

\begin{lemma}\label{lemma gamma(E)=gamma(A)-gamma(C)}
	If $E$ is a gLM bundle, then one has $k=\gamma(E)\le \gamma(A)-\gamma(C)$.
\end{lemma}
\begin{proof}
	As in the proof of \cite[Proposition 3.17]{haburcak_2022}, we note that $h^0(C,\det(E)), h^1(C,\det(E))\ge 2$. Hence $\det(E)\vert_C$ contributes to the Clifford index of $C$, and thus $\gamma(\det(E)\vert_C)\ge \gamma(C)$. Since $\gamma(\det(E)\vert_C)=\gamma(A)-\gamma(E)-2h^1(S,N)$, the result follows.
\end{proof}

\begin{prop}\label{prop gamma(E) bounded}
	Suppose $E$ is a gLM bundle. If $\rho(g,r^\prime, d^\prime)\ge 0$, then \[\gamma(E)\le\gamma(A) +r^\prime -\frac{r^\prime g}{r^\prime+1}.\]
\end{prop}
\begin{proof}
	Let $k=\gamma(E)$, then $\det(E)$ is of type $\g{r^\prime}{\gamma(A)-k+2r^\prime}$. The bound is obtained from \[\rho(g,r^\prime,\gamma(A)-k+2r^\prime)\ge 0.\]
	
	We compute 
	\begin{align*}
		\rho(g,r^\prime,\gamma(A)-k+2r^\prime)&\ge 0\\
		\Longleftrightarrow k &\le \frac{g}{r^\prime +1} +r^\prime -g+\gamma(A)\\
		&= \gamma(A) +r^\prime -\frac{r^\prime g}{r^\prime+1}.
\qedhere
	\end{align*}
\end{proof}

\begin{cor}\label{cor gamma(E) bounded when BN special}
	Let $E$ be a gLM bundle. If $\gamma(E)>\gamma(A) +r^\prime -\frac{r^\prime g}{r^\prime+1}$, then $\langle H , \det(E) \rangle$ is a Brill--Noether special marking on $(S,H)$.
\end{cor}

\section{Brill--Noether special K3 surfaces}\label{Section BN special k3s}

In any genus, the lifting results of Lelli-Chiesa \cite[Theorem
4.2]{Lelli_Chiesa_2015} and Knutsen \cite[Lemma 8.3]{Knutsen2001}
suffice to verify \Cref{Mukai Conj} when $\gamma(C)\le \lfloor
\frac{g-1}{2}\rfloor$ and $\gamma(A)=\gamma(C)$. In those cases,
taking $L$ to be the lift of $A$ provides a Brill--Noether special
marking on $(S,H)$. However, in genus $\ge 14$, there are
non-computing line bundles, which have $\gamma(A)>\gamma(C)$.
We verify \Cref{Mukai Conj} in genus $14-19$ using Donagi--Morrison lifts. For each lattice $\Lambda^r_{g,d}$ associated with the Brill--Noether special NL divisors $\mathcal{K}^r_{g,d}$, we aim to identify $\Lambda^r_{g,d}$ as coming from a (Donagi--Morrison) lift of a Brill--Noether special linear system on a curve $C\in|H|$.

\subsection{Strong DM Holds for genus \texorpdfstring{$14-19$}{}}\label{subsection strong dm holds}
In genus $\ge 14$, there are non-computing line bundles, and thus the lifting results of Lelli-Chiesa and Knutsen \cite{Lelli_Chiesa_2015,Knutsen2001} where the line bundle is assumed to compute $\gamma(C)$ do not suffice. 

We list the non-computing line bundles in genus $14-19$.
\begin{itemize}
	\item $g=14$: $\g{2}{11}$, $\g{3}{13}$;
	\item $g=15$: $\g{3}{14}$;
	\item $g=16$: $\g{2}{12}$, $\g{3}{14}$;
	\item $g=17$: $\g{2}{13}$, $\g{3}{15}$;
	\item $g=18$: $\g{2}{13}$, $\g{3}{15}$, $\g{3}{16}$, $\g{4}{17}$;
	\item $g=19$: $\g{2}{14}$, $\g{3}{16}$, $\g{3}{17}$, $\g{4}{18}$.
\end{itemize}

To verify \Cref{Mukai Conj} in genus $14-19$, we first verify that \Cref{conj Strong DM} holds for the non-computing line bundles, where we use the lifting results of Lelli-Chiesa when $r=2$ \cite{Lelli_Chiesa_2013}, and our previous work when $r=3$ \cite{haburcak_2022}. We then use the Donagi--Morrison lifts to show that $\Pic(S)$ has a Brill--Noether special marking.

\begin{lemma}\label{lemma large slope or BN special}
	Let $(S,H)$ be a K3 surface of genus $g$. Suppose there is a line bundle $L\in\Pic(S)$ of type $\g{r}{d}$ with $L$ and $H-L$ globally generated. If $d\le r$, then $(S,H)$ is Brill--Noether special.
\end{lemma}
\begin{proof}
	Let $r=\frac{2+L^2}{2}\ge 1$, and $d\le r$. We compute $$\rho(g,r,d)=g-(r+1)(g-d+r)=-rg-(r+1)(r-d)<0.$$ Thus the sublattice $\langle H, L \rangle \subseteq \Pic(S)$ is a Brill--Noether special marking of $\Pic(S)$.
\end{proof}

\begin{remark}
	Thus, to show that $(S,H)$ is Brill-Noether special if $C\in |H|$ has a Brill--Noether special $\g{3}{d}$, we may assume $\mu\ge 4$ in \Cref{theorem lifting g3ds general}.
\end{remark}

For each of the non-computing $\g{r}{d}$s in genus $14-19$, we show that $(S,H)$ is Brill--Noether special and/or \Cref{conj Strong DM} holds for a basepoint free $\g{3}{d^\prime}$, which allows us to use known lifting results.

\begin{lemma}\label{Lemma strong DM holds}
	Let $(S,H)$ be a polarized K3 surface of genus $14\le g \le 19$ and $C\in|H|$ a smooth irreducible curve. Suppose $A$ is a primitive non-computing line bundle on $C$, then at least one of the following holds:
	\begin{enumerate}[label=\normalfont(\alph*)]
		\item $(S,H)$ is Brill--Noether special;
		\item there is a line bundle $N\hookrightarrow E_{C,A}$ such that $h^0(S,N)\ge 2$ and $E=E_{C,A}/N$ is stable, that is, \Cref{conj Strong DM} holds; or,
		\item $A$ is of type $\g{4}{d}$, and $C$ has a basepoint free complete linear system of type $\g{3}{d^\prime}$ with $d^\prime \le d-1$. Moreover, {\normalfont(a)} holds or {\normalfont(b)} holds for the $\g{3}{d^\prime}$.
	\end{enumerate}
\end{lemma}
\begin{proof}	
	We argue by $r(A)$. In particular, since $A$ is non-computing, we have $r(A)\ge 2$.
	
	If $A$ is of type $\g{2}{d}$, we are always in case (b), the proof follows from \cite{Lelli_Chiesa_2013}.
	
	If $A$ is of type $\g{3}{d}$, we apply \Cref{theorem lifting g3ds general} and \Cref{lemma large slope or BN special}. In the proof of \cite[Lemma 4.7]{haburcak_2022}, we cannot have $m=0$, as then $c_1(M_1)^2=0$ in \cite[Lemma 4.6]{haburcak_2022}, and $M_1$ is a gLM of type (II) and if $c_1(M_1)^2=0$, then $c_2(M_1)=0$, which is not the case as $M_1$ is stable and thus has $c_2(M_1)\ge 2$, as stated before the proof of \cite[Lemma 4.6]{haburcak_2022}. Likewise, by \Cref{lemma large slope or BN special}, we see that either $(S,H)$ is Brill--Noether special and we are in case (a) or else $\mu$ in \Cref{theorem lifting g3ds general} may be assumed large and we are in case (b).
	
	If $A$ is of type $\g{4}{d}$ as in genus $18$ or $19$, we show we are in case (c). In genus $18$, if $A$ is of type $\g{4}{17}$, then $C$ has a $\g{3}{16}$ by subtracting a non-basepoint from the $\g{4}{17}$, see \cite{Farkas_2001,Lelli-Chiesa_the_gieseker_petri_divisor_g_le_13}. If the $\g{3}{16}$ is not basepoint free, then subtracting basepoints yields the result. Likewise in genus $19$, if $A$ is of type $\g{4}{18}$, then $C$ has a basepoint free $\g{3}{d^\prime}$ with $d^\prime \le 17$. Applying the above argument to the $\g{3}{d^\prime}$ yields the last statement of case (c).
\end{proof}	

\begin{remark}\label{remark obtaining bp free g^3_d}
	The argument above of obtaining a $\g{3}{d^\prime}$ by subtracting a non-basepoint from a $\g{4}{d}$ is the trivial containment of Brill--Noether loci $\mathcal{M}^r_{g,d}\subseteq \mathcal{M}^{r-1}_{g,d-1}$ when $\rho(g,r-1,d-1)<0$, see \cite{Farkas_2001,Lelli-Chiesa_the_gieseker_petri_divisor_g_le_13}. This containment no longer applies in genus $\ge 20$ when the $\g{3}{d-1}$ is no longer Brill--Noether special, i.e. when $\mathcal{M}^4_{g,d}$ is expected maximal ($\rho(g,4,d)<0$, $\rho(g,4,d+1)\ge 0$, and $\rho(g,3,d-1)\ge 0$). For details on expected maximal Brill--Noether loci, see \cite{haburcak_2022}.
\end{remark}

\begin{theorem}
\label{Theorem bn special k3s in genus 14-19}
Let $(S,H)$ be a polarized K3 surface of genus $14\le g \le 19$ and
$C\in|H|$ a smooth irreducible curve. Then $C$ is Brill--Noether
special if and only if $(S,H)$ is Brill--Noether special.
\end{theorem}
\begin{proof}
If $(S,H)$ is Brill--Noether special, with a Brill--Noether special
marking $\langle H,L \rangle$, then taking $L\vert_C$ shows that $C$
is Brill--Noether special.
	
	Conversely, suppose that $C$ is Brill--Noether special and has a line bundle $A$ with $\rho(C,A)<0$. We argue by the Clifford index of $C$. 
	
	If $\gamma(C)<\lfloor\frac{g-1}{2}\rfloor$, or if $\gamma(C)=\lfloor\frac{g-1}{2}\rfloor$ and $\gamma(A)=\gamma(C)$, then the lifting results of Lelli-Chiesa \cite[Theorem 4.2]{Lelli_Chiesa_2015} and Knutsen \cite[Lemma 8.3]{Knutsen2001} give a Brill--Noether special marking on $\Pic(S)$.
	
	Now suppose that $\gamma(C)=\lfloor\frac{g-1}{2}\rfloor$, and $\gamma(A)>\gamma(C)$, that is, $A$ is non-computing. We may assume $A$ is primitive, as otherwise subtracting basepoints from $A$ or $\omega_C\otimes A^\vee$ would give a Brill--Noether special line bundle with lower Clifford index. If $A$ is of type $\g{4}{d}$, we may assume that the $\g{3}{d^\prime}$ obtained from $A$ is complete as else it would have Clifford index lower than $A$. Enumerating all of the Donagi--Morrison lifts of $A$ (or of a $\g{3}{d^\prime}$ obtained from $A$) constrained by \Cref{prop dm lift restrictions if quotient is stable}, \Cref{lemma gamma(E)=gamma(A)-gamma(C)}, \Cref{cor gamma(E) bounded when BN special}, and \Cref{Lemma strong DM holds} show that either $(S,H)$ is Brill--Noether special or the Donagi--Morrison lifts of $A$ give a Brill--Noether special marking on $\Pic(S)$, as desired.
\end{proof}

\begin{remark}
	We give an example of how \Cref{prop dm lift restrictions if quotient is stable} is used in the last statement. Let $g=18$, and $C$ have a $\g{4}{17}$. Suppose that \Cref{Lemma strong DM holds}(c) gives a basepoint free $\g{3}{16}$, which we denote by $A$. Let $E$ be the stable quotient of $E_{C,A}$ by a nontrivial line bundle $N$, as in \Cref{coker is gLM}. By \Cref{prop dm lift restrictions if quotient is stable} and \Cref{remark dm lift restrictions if quotient is stable}, we see that $\det(E)$ must be of type $\g{3}{16}$, $\g{4}{17}$, $\g{4}{16}$, $\g{5}{18}$, or $\g{6}{20}$, all of which induce Brill--Noether special markings on $\Pic(S)$. The argument is similar if $A$ is a $\g{3}{d^\prime}$ with $d\le 15$.
\end{remark}


	\providecommand{\bysame}{\leavevmode\hbox to3em{\hrulefill}\thinspace}
\providecommand{\MR}{\relax\ifhmode\unskip\space\fi MR }
\providecommand{\MRhref}[2]{%
	\href{http://www.ams.org/mathscinet-getitem?mr=#1}{#2}
}
\providecommand{\href}[2]{#2}

	\vfill
\end{document}